\documentclass[11pt]{article}
\usepackage[top=3cm, bottom=3cm, left=3cm, right=3cm]{geometry}
\usepackage{amsmath,amsthm,amssymb,amsfonts}

\usepackage{amsmath,amsthm,amssymb,graphicx, multicol, array}
\usepackage{array}
\usepackage{enumerate}
\usepackage{cases}

\usepackage{hyperref}

\DeclareMathOperator{\ord}{ord}

\newtheorem{thm}{Theorem}
\newtheorem{lem}{Lemma}

\newtheorem{rem}{Remark}

\title{Least Prime Primitive Roots}
\date{}
\author{N. A. Carella}

\begin{document}
\maketitle

\vskip .25 in 
\textbf{Abstract:} This note presents an upper bound for the least prime primitive roots \(g^*(p)\) modulo \(p\), a large prime. The current 
literature has several estimates of the least prime primitive root \(g^*(p)\) modulo a prime \(p\geq 2\) such as \(g^*(p)\ll  p^c,
c>2.8\). The estimate provided within seems to sharpen this estimate to the smaller estimate \(g^*(p)\ll p^{5/\log \log p}\) uniformly for all 
large primes \(p\geq 2\).  \\

\tableofcontents

\vskip .25 in 


\textbf{AMS Mathematical Subjects Classification:} Primary 11A07, Secondary 11Y16, 11M26.\\

\newpage
\section{Introduction} \label{s1}
This note provides the details for the analysis of some estimates for the least primitive root \(g(p)\), and the least prime primitive root \(g^*(p)\) in the cyclic group \(\mathbb{Z}/(p-1)\mathbb{Z}, p\geq 2\) prime. The current literature has several estimates for the least prime primitive root \(g^*(p)\) modulo a prime \(p\geq 2\) such as \(g^*(p)\ll  p^c, c>2.8\). The actual constant \(c>2.8\) depends on various conditions such
as the factorization of \(p-1\), et cetera. These results are based on sieve methods and the least primes in arithmetic progressions, see \cite{MG97}, \cite{MG98}, \cite{JH13}. Moreover, there are a few other conditional estimates such as \(g(p)\leq g^*(p)\ll (\log  p)^6\), see \cite{VS92}, and the conjectured upper bound \(g^*(p)\ll (\log  p)(\log \log p)^2\), see \cite{BE97}. On the other direction, there is the Turan lower bound \(g^*(p)\geq g(p)=\Omega (\log p \log \log p)\), refer to \cite{BD71}, \cite[p.\ 24]{RP96}, and \cite{ML91} for discussions. The result stated in Theorem \ref{thm1.1} improves the current estimate to
the smaller estimate \(g^*(p)\ll p^{5/\log \log p}\) uniformly for all large primes \(p\geqslant 2\). \\

\begin{thm} \label{thm1.1}  Let \(p\geq 3\) be a large prime. Then the following hold.
\begin{enumerate} 
\item Almost every prime \(p\geq 3\) has a prime primitive root \(g^*(p)\ll (\log p)^c, c>1\) constant.
\item Every prime \(p\geq 3\) has a prime primitive root \(g^*(p)\ll p^{5/\log\log 
p}\).
\end{enumerate}
\end{thm}

Case (1) explains the frequent occurrence of very small primitive roots for almost every prime; and case (2) explains the rare occurrence of large prime primitive roots modulo \(p\geq 2\) on a subset of density zero in the set of primes. The term \textit{for almost every prime} refers
to the set of all primes, but a subset of primes of zero density. The subset of exceptional primes are of the form \(p-1=\prod _{q\leq \log p} q^v\), where \(q\geq 2\) is prime, and \(v\geq 1\). The proof appears in Section \ref{s5}.

\section{Basic Concepts} \label{s2}
Let \(G\) be a finite group of order \(q=\# G\). The order \(\ord(u)\) of an element \(u\in G\) { }is the smallest integer \(d \,|\,  q\) such that \(u^d=1\). An element \(u\in G\) is called a \textit{primitive element} if it has order \(\ord(u)=q\). A cyclic group \(G\) is a group generated by a primitive element \(\tau \in G\). Given a primitive root \(\tau \in G\), every element \(0\neq u\in G\) in a cyclic group has
a representation as \(u=\tau ^v, 0\leq v<q\). The integer \(v=\log  u=\log_{\tau} u\) is called the \textit{discrete logarithm} of \(u\) with respect to \(\tau\).\\

\subsection{Simple Characters Sums}
Let $d\,|\,q$. A character \(\chi\) modulo \(q\geq 2\), is a complex-valued periodic function \(\chi :\mathbb{N} \longrightarrow  \mathbb{C}\), and it has order \(\ord(\chi ) = d\geq 1\) if and only if $\chi $\((n)^d=1\) for all integers \(n \in  \mathbb{N}, \gcd (n,q)=1\). For \(q \neq 2^r,r\geq 2\), a multiplicative character \(\chi \neq 1\) of order \(\ord(\chi )=d\), has a representation as
\begin{equation}
\chi(u)=e^{i2\pi  k \log (u)/d },
\end{equation}
where \(v=\log  u\) is the discrete logarithm of \(u\neq 0\) with respect to some primitive root, and for some integer \(k\in \mathbb{Z}\), see \cite[p. 187]{LN97}, \cite[p.\ 118]{MV07}, and \cite[p.\ 271]{IK04}. The principal character \(\chi _0=1\) mod \(q\) has order \(d =1\), and it is defined by the relation
\begin{equation}
\chi _0(n)=
\left \{  
\begin{array}{ll}
 1 & \text{  }\text{if } \gcd (n,q)=1, \\
 0 & \text{  }\text{if } \gcd (n,q)\neq 1. 
\end{array} \right .
\end{equation}
And the nonprincipal character \(\chi \neq 1\) mod \(q\) of order \(\ord(\chi )=d>1\) is defined by the relation
\begin{equation}
\chi(n)=
\left \{  
\begin{array}{ll}
\omega^{\log n} & \text{  if } \gcd(n,q)=1,\\

0 & \text{  if } \gcd (n,q)\neq 1, 
\end{array} \right .
\end{equation} 

where \(\omega \in \mathbb{C}\) is a \textit{d}th root of unity.\\ 

The Mobius function and Euler totient function occur in various formulae. For an integer $n=p_1^{v_1}p_2^{v_2}\cdots  p_t^{v_t}$, with $p_k\geq 2$ prime, and $
v_i\geq 1$, the Mobius function is defined by \\
\begin{equation}
\mu (n)=\left \{
\begin{array}{ll}
 (-1)^t & \text{if } n =p_1p_2\cdots  p_t,\text{               }v_k=1 \text{ all } k\geq 1,  \\
 0 & \text{if } n \neq p_1p_2\cdots  p_t,\text{   }v_k\geq 2\text{ for some } k\geq 1. 
\end{array} \right.
\end{equation}\\

The Euler totient function counts the number of relatively prime integers \(\varphi (n)=\#\{ k:\gcd (k,n)=1 \}\). This is compactly expressed by
the analytic formula 
\begin{equation}
\varphi (n)=n\prod_{d\,|\,n }(1-1/p) =n\sum _{d \,|\, n} \frac{\mu (d)}{d}.
\end{equation}

\begin{lem} \label{lem2.1}  For a fixed integer \(u\neq 0\), and an integer \(q\in \mathbb{N}\), let \(\chi \neq 1\) be nonprincipal character mod \(q\), then 
\begin{enumerate}
\item $\displaystyle 
\sum _{ \ord(\chi) = \varphi (q)} \chi (u)= 
\left \{
\begin{array}{ll}
\varphi (q)     &\text{ if } u\equiv 1\text{ mod } q,\\
-1 & \text{ if } u\not\equiv 1\text{ mod } q.\\
\end{array}
\right .$

\item $\displaystyle \sum _{ 1\leq a<\varphi (q)} \chi (a u)=\left \{
\begin{array}{ll}
\varphi (q)     &\text{ if } u\equiv 1\text{ mod } q,\\
-1 & \text{ if } u\not \equiv1\text{ mod } q.\\
\end{array}
\right . $
\end{enumerate}
\end{lem}

\subsection{Representation of the Characteristic Function}
The characteristic function \(\Psi :G\longrightarrow \{ 0, 1 \}\) of primitive element is one of the standard tools employed to investigate the various
properties of primitive roots in cyclic groups \(G\). Many equivalent representations of characteristic function $\Psi $ of primitive elements are
possible. The best known representation of the characteristic function of primitive elements in finite rings is stated below. 
\\

\begin{lem} \label{lem2.2} Let \(G\) be a finite group of order \(q=\# G\), and let \(0\neq u\in G\) be an invertible element of the
group. Assume that \(v=\log u\), and \(e=\gcd (d,v).\) Then \\
\begin{equation}
\Psi (u)=\frac{\varphi (q)}{q}\sum _{d \, | \,q} \frac{\mu(d)}{\varphi
(d)}\sum_{\ord(\chi)= d} \chi (u)=
\left
 \{  
\begin{array}{ll}
 1 & \text{if } \ord(u) =q,  \\
 0 & \text{if } \ord(u) \neq q. \\
\end{array} \right .
\end{equation}\\
\end{lem}
Finer details on the characteristic function are given in \cite[p.\ 863]{ES57}, \cite[p.\ 258]{LN97}, \cite[p.\ 18]{MP04}, \cite{WA01}, et alii. The characteristic function for multiple primitive roots is used in \cite[p.\ 146]{CZ01} to study consecutive primitive roots. In \cite{KS12} it is used to study the gap between primitive roots with respect to the Hamming metric. In \cite{SI11} it is used to study Fermat quotients as primitive roots. And in \cite{WA01} it is used to prove the existence of primitive roots in certain small subsets \(A\subset \mathbb{F}_{p^n},n\geq 1\). Many other applications are available in the literature. An introduction to elememtary character sums as in Lemmas \ref{lem2.1} and \ref{lem2.2} appears in \cite[Chapter 6]{LN97}.\\

\section{Basic \textit{L}-Functions Estimates} \label{s3}
For a character \(\chi\) modulo \(q\geq 2\), and a complex number \(s\in \mathbb{C}, \mathcal{R}e(s)>1\), an  \textit{L}-function is defined by the infinite sum \(L(s,\chi )=\sum_{ n \geq 1} \chi (n)n^{-s}\). Simple elementary estimates associated with the characteristic function of primitive roots are calculated here.\\

\subsection{An \textit{L}-Function for Prime Primitive Roots}
The analysis of the least prime primitive root \(\text{mod } p\) is based on the Dirichlet series 
\begin{equation}
L(s,\Psi \Lambda )=\sum _{n \geq  1}
\frac{\Psi (n)\Lambda (n)}{n^s} ,
\end{equation}
where \(\Psi (n)\Lambda (n)\left/n^s\right.\) is the weighted characteristic function of primes and prime powers primitive roots \(\text{mod } p\). This is constructed using the characteristic function of primitive roots, which is defined by 
\begin{equation}
\Psi (n)=
\left \{  
\begin{array}{ll}
 1 & \text{if } n \text{ is a primitive root}, \\
 0 & \text{if } n \text{ is not a primitive root}, 
\end{array} \right .
\end{equation}
where \(p\geq 2\) is a prime, see Lemma 3 for the exact formula, and the vonMangoldt function 
\begin{equation}
\Lambda (n)=
\left \{ 
\begin{array}{ll}
 \log  p & \text{if } n=p^k,k\geq 1, \\
 0 & \text{if } n\neq p^k,k\geq 1. 
\end{array}
 \right .
\end{equation}

The function \(L(s,\Psi \Lambda )\) is zerofree, and analytic on the complex half plane \(\{ s\in \mathbb{C}: \text{$\mathcal{R}$e}(s)=\sigma >1
\}\). Furthermore, it has a pole at \(s=1\). This technique has a lot of flexibility and does not require delicate information on the zerofree regions
\(\{ s\in \mathbb{C}: 0<\text{$\mathcal{R}$e}(s)=\sigma <1 \}\) of the associated \textit{ L}-functions. Moreover, the analysis is much simpler than
the sieve methods used in the current literature -- lex parsimoniae.\\

The nonprincipal partial sum of the previous $L$-function has a mild dependence on the prime $p$, this is made explicit in the next result.

\begin{lem} \label{lem3.1}  Let \(p\geq 2\) be a prime number, and let $\chi_0$ be the principal character modulo $p$. For a real number $x \geq 1$, and a real number \(s=\sigma > 1\), the nonprincipal partial sum
\begin{equation}
\sum _{n \leq  x} \frac{\Lambda (n)\chi_0(n)}{n^s} =- \frac{\zeta^{'}(s)}{\zeta(s)}-
\sum _{k \geq  0} \frac{\log p}{p^{(1+k)s}}+O\left( \frac{\log x}{x^{s-1}}\right) .
\end{equation}
\end{lem}

\begin{proof} The principal character $\chi_0(n)=0$ if and only if $p\mid n$. Otherwise $\chi(n)=1$. Moreover, the series 
\begin{equation}
\sum _{n \geq  1} \frac{\Lambda (n)}{n^s} =- \frac{\zeta^{'}(s)}{\zeta(s)},
\end{equation}	
where $\zeta(s)$ is the zeta function, is absolutely convergent for $\mathcal{R}e(s)=\sigma>1$. The smaller correction constant
\begin{equation}
\sum _{n \geq  1} \frac{\Lambda (pn)}{(pn)^s} =\sum _{k \geq  0} \frac{\log p}{p^{(1+k)s}}
\end{equation}
accounts for the missing integers $n \equiv 0 \bmod p$. In particular, at $s=2$, this is  
\begin{equation} \label{300}
- \frac{\zeta^{'}(2)}{\zeta(2)}-
\sum _{k \geq  0} \frac{\log p}{p^{(1+k)2}}+O\left( \frac{\log x}{x}\right) =\kappa_2(p) +O\left( \frac{\log x}{x}\right),
\end{equation}
where $\kappa_2(p)>0$ is a constant for any fixed prime $p$. 
\end{proof}
A formula for computing the constant $-\zeta^{'}(2)/\zeta(2)>0$ is given in \cite[25.6.15]{DLMF}.

\begin{lem} \label{lem3.2}  Let \(x\geqslant 2\) be a real number. Let \(p\geq 2\) be a prime, and let \(\chi\ne 1\) be the nontrivial characters modulo \(d\), where \(d \mid (p-1)\). 
If \(s=\sigma >1\) is a real number, then principal partial sum
\begin{equation}
\sum _{n\leq x} \frac{\Lambda (n)}{n^s}\sum _
{1<d\mid p-1} \frac{\mu (d)}{\varphi (d)}\sum _{\text{ord}(\chi ) = d}  \chi (n)=O\left(\frac{2^{\omega (p-1)}}{x^{\sigma
-1}}\right) .
\end{equation}
\end{lem}

\begin{proof} Rearranging the principal partial sum and taking 
absolute value yield 
\begin{eqnarray} \label{332}
\left| \sum _{n \leq x} \frac{\Lambda (n)}{n^s}\sum_{1<d \mid
p-1} \frac{\mu(d)}{\varphi(d)}\sum _{\text{ord}(\chi )= d} \chi (n)\right| &=&\left| \sum_{1<d \mid p-1}\frac{\mu (d)}{\varphi
(d)}\sum _{\text{ord}(\chi ) = d,}  \sum _{n \leq x} \frac{\chi (n)\Lambda (n)}{n^s} \right| \nonumber \\ 
&\leq & \sum_{1<d \mid p-1} \left |\frac{\mu (d)}{\varphi
(d)} \right |  \left |\sum _{\text{ord}(\chi ) = d,}  \sum _{n \leq x} \frac{\chi (n)\Lambda (n)}{n^s} \right| \nonumber \\ 
&\leq& \sum _{1<d \mid p-1} \mu ^2(d) \left|
\sum_{n \leq x} \frac{\chi (n)\Lambda (n)}{n^s} \right| ,
\end{eqnarray} 
where $\varphi(d)=\#\{ \chi \ne 1: \chi^d=1 \}$ is the number of nontrivial characters of order $d>1$. The inner sum is estimated, using Abel summation formula, (discussed in \cite[p. 12]{MV07} and \cite{TG95}, as 
\begin{equation} \label{333}
\sum _{n \leq x} \frac{\chi (n)\Lambda
(n)}{n^s}=\int _1^{x }\frac{1}{t^s}d \psi _{\chi }(t) \ll \frac{1}{x^{s-1}},
\end{equation}
where 
\begin{equation} \label{334}
\left |  \psi _{\chi }(x) \right| =\left| \underset{n \leq  x}\sum \chi (n)\Lambda (n) \right| \ll x . 
\end{equation}
Substitute this estimate of the integral back into (\ref{332}) returns
\begin{eqnarray}
\sum _{1<d \mid p-1} \mu ^2(d) \left|
\sum_{n \leq x} \frac{\chi (n)\Lambda (n)}{n^s} \right| 
&\ll& \frac{1}{x^{\sigma-1}} \sum _{1<d \mid p-1} \mu ^2(d) \\ 
&\ll& O \left (\frac{2^{\omega(p-1)}}{x^{\sigma-1}} \right ) \nonumber.
\end{eqnarray}
 \end{proof}

The estimates in (\ref{333}) and (\ref{334}) are trivial, but sufficient. Similar and sharper results are proved in \cite[Theorem 4.11]{MV07}, and similar references.

\section{Prime Divisors Counting Function} \label{s4}
For \(n\in \mathbb{N}\), the prime counting function is defined by \(\omega (n)=\#\{ p \mid n \}\). Somewhat similar proofs of the various properties of the arithmetic function \(\omega (n)\) are given in \cite[p.\ 473]{HW08}, \cite[p.\ 55]{MV07}, \cite[p.\ 34]{CM06}, \cite[p.\ 83]{TG95}, and other references. 

\begin{lem} \label{lem4.1}  Let \(n\geq 1\) be a large integer. Then,
\begin{enumerate} 
\item Almost every integer \(n\geq 1\) satisfies the inequality 
\begin{equation}
\omega(n)\ll  \log\log n .
\end{equation}
\item Every integer \(n\geqslant 1\) satisfies the inequality 
\begin{equation}
\omega (n)  \ll  n/\log\log n .
\end{equation}
\end{enumerate} 
\end{lem} 

\begin{lem} \label{lem4.2}  Let \(n\geq 1\) be a large integer. Then, 
\begin{enumerate}
\item Almost every integer \(n\geq 1\) satisfies the number of squarefree divisors inequality 
\begin{equation}
2^{\omega (n)}\ll 2^{\log\log n } .
\end{equation}
\item For every integer \(n\geq 1\), the number of squarefree divisors satisfies the inequality 
\begin{equation}
2^{\omega (n)}\ll 2^{\log  n/\log\log
	n} .
\end{equation}
\end{enumerate}
\end{lem}

The Euler constant is defined by the formula $ \gamma =\lim_{x \rightarrow  \infty } \left(\sum _{ n\leq x } 1/n-\log  x\right), $ and the Mertens constant is defined by the formula $B_1=\lim_{x \rightarrow  \infty } \left(\sum _{ p\leq x } 1/p-\log \log x\right),$ and \(c>0\) is an absolute constant.\\

\begin{rem} { \normalfont The average order of the prime divisor counting function is precisely
\begin{equation} 
\frac{1}{x} \sum_{n \leq x} \omega(n)=  \log\log n +B_1+(\gamma -1)/\log  n+O\left( e^{-c\sqrt{\log n}}\right),
\end{equation}
and the standard deviation is $\ll\sqrt{\log \log n}$, this implies the first statement in Lemma \ref{lem4.1}.  And any integer $n \geq1  $ has at most 
$\omega(n) \ll \log n/ \log \log n$ primes, this can be proved using the product
 \begin{equation}
n=\prod_{p \leq w} p \ll e^{\sum_{p \leq w} \log p} .
\end{equation}
The properties of the distribution of the omega function has been extended to a wide range of subsets of integers, arithmetic functions and numbers fields, the interested reader should refer to the literature. In particular, $\omega(p-1)$ satisfies all the properties stated in Lemmas \ref{lem4.1} and \ref{lem4.2}, see \cite{HH56} and other more recent publications.

} \end{rem}

\section{The Least Prime Primitive Roots} \label{s5}
It is expected that there are estimates for the least prime primitive roots \(g(p)^*\geq g(p)\), which are quite similar to the estimates for the least primitive roots \(g(p)\). This is a very small quantity and nowhere near the currently proved results. In fact, the numerical tables confirm that the least prime primitive roots are very small, and nearly the same magnitude as the least primitive roots, but have more complex patterns,
see \cite{PS02}.\\

\begin{proof} (Theorem \ref{thm1.1}.) Fix a large prime \(p\geq 2\), and consider summing the weighted characteristic function \(\Psi (n)\Lambda (n)/n^s\) over a small 
range of both primes and prime powers \(q^k\leq x, k\geq 1\), see Lemma \ref{lem2.2}. Then, the nonexistence
equation 
\begin{equation}
0=\sum _{n \leq  x} \frac{\Psi (n)\Lambda (n)}{n^s}=\sum _{n \leq  x} \frac{\Lambda (n)}{n^s}\left(\frac{\varphi (p-1)}{p-1}\sum
_{d \,| \,p-1} \frac{\mu (d)}{\varphi (d)}\sum _{\text{ord}(\chi ) = d} \chi (n)\right) ,
\end{equation}\\
where \(s>1\) is a real number, holds if and only if there are no primes or prime powers primitive
roots in the interval \([1, x]\).\\ 

Separating the nonprincipal part and the principal part, and applying Lemma \ref{lem3.1}, and Lemma \ref{lem3.2} with \(q=p-1\), and \(s=2\), yield 
\begin{eqnarray} \label{el50}
0&=&\sum_{n \leq  x} \frac{\Lambda (n) \chi_0(n)}{n^s} + \sum_{1<d \,| \, p-1} \frac{\mu (d)}{\varphi (d)}\sum_{\text{ord}(\chi ) = d} \sum_{n \leq  x} \frac{\Lambda (n) \chi(n)}{n^s} \nonumber\\
&=&\kappa_2+O\left(\frac{\log x}{x}\right)+O\left(\frac{2^{\omega (p-1)}}{x}\right),
\end{eqnarray}
where \(\kappa_2=\kappa _2(p)>0\) is a constant, which depends on the fixed prime \(p\geq 2\), see (\ref{300}).\\

\textbf{Case I:}  Restriction to the average integers $p-1$, with \(2^{\omega (p-1)}\ll \log  p\). Refer to Lemmas \ref{lem4.1} and \ref{lem4.2} for more details.\\

Let \(x=(\log  p)^{1+\varepsilon }, \varepsilon >0\), and suppose that the short interval \([2, (\log  p)^{1+\varepsilon }]\) does not contain
primes or prime powers primitive roots. Then, replacing these information into (\ref{el50}) yield 
\begin{eqnarray}
0&=&\kappa_2+O\left(\frac{\log x}{x}\right)+O\left(\frac{2^{\omega (p-1)}}{x}\right) \\ 
&=&\kappa_2+O\left(\frac{1}{(\log p)^{\varepsilon }}\right) >0 \nonumber.
\end{eqnarray}

Since \(\kappa _2>0\) is a constant, this is a contradiction for all sufficiently large prime \(p\geq 3\).\\ 

\textbf{Case II:}  No restrictions on the integers $p-1$, with \(2^{\omega (p-1)}\ll p^{4/\log\log p}\). Refer to Lemmas  \ref{lem4.1} and \ref{lem4.2} for more details.\\

Let \(x=p^{5/\log\log p}\), and suppose that the short interval \([2, p^{5/\log\log p}]\) does not contain primes or prime powers primitive roots.
Then, replacing these information into (\ref{el50}) yield 
\begin{eqnarray} \label{555}
0&=&\kappa_2+O\left(\frac{\log x}{x}\right)+O\left(\frac{2^{\omega (p-1)}}{x}\right) \\
&=&\kappa _2+O\left(\frac{1}{p^{1/\log\log
p}}\right)>0 \nonumber.
\end{eqnarray}

Since \(\kappa _2>0\) is a constant, this is a contradiction for all sufficiently large prime \(p\geq 3\). 
\end{proof}

\begin{rem} \label{rem5.1} { \normalfont Explaination For the two cases.
\begin{enumerate}
\item   Case I: The average magnitute of $\omega(\ord_p(\tau))=\omega(p-1) \ll \log \log p$. There is an explicit relationship between the order $\ord_p(\tau)=p-1$ of the element $\tau \in \mathbb{F}_p$, the number of primes $\omega(\ord_p(\tau))=\omega(p-1)$, the order of the multiplicatice group $\#G=p-1$, the number of primes $\omega(\#G)=\omega(p-1)$, and the prime $p$. These explicit dependencies lead to the upper limit $x=(\log  p)^{1+\varepsilon }, \varepsilon >0$.

\item Case II.  The extreme size of $\omega(\ord_p(\tau))=\omega(p-1)\ll \log p$. There is an explicit relationship between the order $\ord_p(\tau)=p-1$ of the element $\tau \in \mathbb{F}_p$, the number of primes $\omega(\ord_p(\tau))=\omega(p-1)$, the order of the multiplicatice group $\#G=p-1$, the number of primes $\omega(\#G)=\omega(p-1)$, and the prime $p$. These explicit dependencies lead to the upper limit $x=p^{5/\log\log p}$.
\end{enumerate}
} \end{rem}

\section{Addendum} \label{s6}
The same analysis extends to other multiplicative groups $G$ of order $\#G=q$ provided that the order $\ord_q(\tau)$ of the element $\tau \in G$, the number of primes $\omega(\ord_q(\tau))$, the number of primes $\omega(\#G)=\omega(q)$, and the integer $q$ have a explicit nontrivial relationship.\\

 For elements in the multiplicative group $\left (\mathbb{Z}/p\mathbb{Z} \right )^{\times}$ such that $\omega(\ord_p(\tau))$ is a small constant, there is no explict relationship between the prime $p$ and $x$, accorddingly, the analysis is very delicate. \\

\subsection{Quadratic Nonresidues} In the case of quadratic nonresidues $n(p) \bmod p$, order $\ord_p(n(p))=2$, and the number of primes $\omega(\ord_p(n(p))=1$. Thus, there is no nontrivial explicit relationship between the different paramters. Moreover, there are just two characters: the principal $\chi_0(n)=1$ and the quadratic character is $\chi(n)=\left ( \frac{n}{p} \right )$. Using the same analysis as in the proof of Theorem \ref{thm1.1}, the last equation (\ref{555}) reduces to
\begin{eqnarray}
0&=&\sum _{n \leq  x} \frac{\left(\chi_0(n) - \chi(n) \right )\Lambda (n)}{n^2}\nonumber\\
&=&\kappa _2(p)+\sum _{n \leq  x} \frac{\chi(n) \Lambda (n)}{n^2}+O\left(\frac{\log x}{x}\right) .
\end{eqnarray}
But, there is no simple expression as in Case I and II above to link the variable $x=x(p)$ and $p$ for all primes $p$, see Remark \ref{rem5.1}. \\

The least $x\geq 1$ is not arbitrary, it depends on the maximal number of consecutive quadratic residues, this made explicit in the decomposition of the $L$-series as
\begin{equation}
\sum _{n \geq  1} \frac{\left(\chi_0(n) - \chi(n) \right )\Lambda (n)}{n^2}=  \sum _{n \geq  1} \frac{\chi_0(n)\Lambda (n)}{n^2} -\sum _{n \geq  1} \frac{\chi(n)\Lambda (n) }{n^2}.
\end{equation}
Assuming the RH, it is $n(p) \leq x=O(\log^2 p)$, see \cite{AN52}. Furthermore, the restriction to certain characters $\chi \mod p$ produces a better bound $n(p)\leq x=(\log p)^{1.4}$ as proved in \cite{BG13}. \\

The classical quadratic nonresidue test 
\begin{equation}
a^{(p-1)/2} \not \equiv 1 \bmod p 
\end{equation}
for $\gcd(a,p)=1$, 
and the primitive root test 

\begin{equation}
a^{(p-1)/r} \not \equiv 1 \bmod p
\end{equation}
 for all prime divisors $r \mid p-1$, are the same or 
nearly the same for primes with very few prime divisors $r \mid p-1$. In the extreme cases of Fermat primes $p=2^{2^m}+1$ and Germain primes $p=2^a q^b+1$ with $a, b \geq 1$ 
and $q\geq 3$ prime, these tests are the same. Consequently, Theorem \ref{thm1.1} for $\omega(p-1) \ll \log \log p$ seems to imply $n(p) \ll \log^{1 + \varepsilon} p$. 

\subsection{Average And Conjectures}
The average value of a quadratic nonresidue $n(p) \mod p$ is extremely small 
\begin{equation}
\overline{n(p)}=\lim_{x \to \infty} \frac{1}{\pi(x)} \sum_{p \leq x} n(p) =2.920050 \ldots,
\end{equation}
a general discussion and a proof appear in \cite{PP12}. But, there exists a subset of primes of zero density that have large nonquadratic residues $n(p) \to \infty$ as $p \to \infty$.\\

The Vinogradov quadratic nonresidue conjecture states that the least quadratic nonresidue modulo $p$ satisfies $n(p) \ll p^{\varepsilon}$, with $\varepsilon >0$ an arbitrary small number, see \cite{KS13}. Since
\begin{equation}
n(p) \leq g^{*}(p) \ll p^{5/\log \log p} \ll p^{\varepsilon},
\end{equation}  
this is implied by Theorem \ref{thm1.1}. In addition, a recent conjecture calls for 
\begin{equation}
n(p)\leq x=O((\log p)(\log \log p)), 
\end{equation}
see \cite{GR90}, \cite{TE17}, and similar literature. \\


\begin{thebibliography}{99}

\bibitem{AN52} Ankeny, N. C. The least quadratic non residue. Ann. of Math. (2)  55,  (1952). 65-72.

\bibitem{AC14} Ambrose, Christopher. On the least primitive root expressible as a sum of two squares. Mathematisches Institut, Universitat Gottingen,
PhD Thesis, 2014.

\bibitem{BE97} Bach, Eric. Comments on search procedures for primitive roots. Math. Comp. 66, (1997), no. 220, 1719-1727. 

\bibitem{BD71} Burgess, D. A. The average of the least primitive root modulo $p>2$. Acta Arith. 18, (1971), 263-271. 

\bibitem{BD68} Burgess, D. A.; Elliott, P. D. T. A. The average of the least primitive root. Mathematika 15, 1968, 39-50.

\bibitem{BD62} Burgess, D. A. On character sums and primitive roots. Proc. London Math. Soc. (3) 12, 1962, 179-192.

\bibitem{BG13} Jonathan Bober, Leo Goldmakher, Polya-Vinogradov and the least quadratic nonresidue, arXiv:1311.7556. 

\bibitem{CZ01} Cobeli, C. I.; Gonek, S. M.; Zaharescu, A. On the distribution of small powers of a primitive root. J. Number Theory 88, (2001), no. 1, 49-58.

\bibitem{CM06} Cojocaru, Alina Carmen; Murty, M. Ram. An introduction to sieve methods and their applications. London Mathematical Society Student Texts, 66. Cambridge University Press, Cambridge, 2006.

\bibitem{DLMF} NIST Digital Library of Mathematical Functions, dlmf.nist.gov.

\bibitem{EP97} Elliott, P. D. T. A.; Murata, Leo. On the average of the least primitive root modulo p . J. London Math. Soc. (2) 56, (1997), no. 3, 435-454.

\bibitem{ES57} Paul Erdos And Harold N. Shapiro, On The Least Primitive Root Of A Prime, 1957, euclidproject.org.

\bibitem{GR90} Graham, S. W.; Ringrose, C. J. Lower bounds for least quadratic nonresidues. Analytic number 
theory (Allerton Park, IL, 1989), 269-309, Progr. Math., 85, Birkh�user Boston, Boston, MA, 1990. 


\bibitem{HH56} Halberstam, H. On the distribution of additive number-theoretic functions. III. J. London Math. Soc. 31 (1956), 14-27.


\bibitem{JH13} Ha, Junsoo. On the least prime primitive root. J. Number Theory 133, (2013), no. 11, 3645-3669. 


\bibitem{HW08} Hardy, G. H.; Wright, E. M. An introduction to the theory of numbers. Sixth edition. Revised by D. R. Heath-Brown
and J. H. Silverman. With a foreword by Andrew Wiles. Oxford University Press, 2008.

\bibitem{IK04} Iwaniec, Henryk; Kowalski, Emmanuel. Analytic number theory.
American Mathematical Society Colloquium Publications, 53. American Mathematical Society, Providence, RI, 2004.\\

\bibitem{KS12} Konyagin, Sergei V.; Shparlinski, Igor E. On the consecutive powers of a primitive root: gaps and exponential sums. Mathematika 58, (2012), no. 1, 11-20.
\bibitem{KS13} Sergei V. Konyagin, Igor E. Shparlinski, Quadratic Non-residues in Short Intervals, arXiv:1311.7016.

\bibitem{LN97} Lidl, Rudolf; Niederreiter, Harald Finite fields. With a foreword by P. M. Cohn. Second edition. Encyclopedia of Mathematics and its Applications, 20. Cambridge University Press, Cambridge, 1997.

\bibitem{MG98} Martin, Greg. Uniform bounds for the least almost-prime primitive root.
Mathematika 45, (1998), no. 1, 191-207. 

\bibitem{MG97} Martin, Greg. The least prime primitive root and the shifted sieve. Acta Arith. 80, (1997), no. 3, 277-288. 

\bibitem{MP04} Pieter Moree. Artin's primitive root conjecture -a survey. arXiv:math/0412262.

\bibitem{ML91} Murata, Leo. On the magnitude of the least prime primitive root. J. Number
Theory 37, (1991), no. 1, 47-66.

\bibitem{MV07} Montgomery, Hugh L.; Vaughan, Robert C. Multiplicative number theory. I. Classical theory. Cambridge
University Press, Cambridge, 2007.

\bibitem{NW00} Narkiewicz, W. The development of prime number theory. From Euclid to Hardy and Littlewood. Springer Monographs in Mathematics. Springer-Verlag,
Berlin, 2000. 

\bibitem{PS02} Paszkiewicz, A.; Schinzel, A. On the least prime primitive root modulo a prime. Math. Comp. 71, (2002), no. 239, 1307-1321.
\bibitem{PP12} Pollack, Paul. The average least quadratic nonresidue modulo m  and other variations on a theme of Erdos. J. Number Theory  132  (2012),  no. 6, 1185-1202.
\bibitem{RP96} Ribenboim, Paulo, The new book of prime number records, Berlin, New York: Springer-Verlag, 1996.

\bibitem{SI11} Igor E. Shparlinski, Fermat quotients: Exponential sums, value set and primitive roots, arXiv:1104.3909.

\bibitem{VS92} Shoup, Victor. Searching for primitive roots in finite fields. Math. Comp.
58, (1992), no. 197, 369-380.

\bibitem{TE17} E. Trevino, The least quadratic non-residue and related problems, http://campus.lakeforest.edu/trevino/CalStateChico.pdf.

\bibitem{TG95} G. Tenenbaum, Introduction to analytic and probabilistic number theory, Cambridge Studies in Advanced Mathematics 46, Cambridge University Press, Cambridge, 1995.

\bibitem{WA01} Winterhof, Arne Character sums, primitive elements, and powers in finite fields. J. Number Theory 91, (2001), no. 1, 153-163.

	
\end{thebibliography}
\end{document}